\documentclass[12pt,regno]{amsart}
\usepackage{amsfonts,amssymb,amsmath}
\usepackage[english]{babel}
\usepackage{amsfonts,amssymb}
\newtheorem{theorem}{Theorem}
\newtheorem{lemma}{Lemma}

\theoremstyle{definition}
\newtheorem{definition}{Definition}
\theoremstyle{remark}
\newtheorem{remark}{Remark}

\theoremstyle{definition}

\newcommand{\Ker}{\mbox{Ker}}

\begin{document}

\title[Sturm-Liouville operators]
    {Sturm-Liouville operators with matrix distributional coefficients}

\author{Alexei Konstantinov}
\address{Taras Shevchenko National University of Kyiv, 64 Volodymyrs'ka, Kyiv, 01601, Ukraine}
\email{konstant12@yahoo.com}

\author{Oleksandr Konstantinov}
\address{Livatek Ukraine LLC, 42 Holosiivskyi Ave., Kyiv, 03039, Ukraine}
\email{iamkonst@ukr.net}

\subjclass[2010]{34L40, 34B08, 47A10}

\keywords{Sturm-Liouville problem, matrix quasi-differential
operator, singular coefficients, resolvent approximation,
self-adjoint extension.}

\begin{abstract}
 The paper deals with the  singular Sturm-Liouville expressions
  $$l(y) = -(py')' + qy$$
  with the matrix-valued coefficients $p,q$ such that
  $$q=Q', \quad  p^{-1},\,\, p^{-1}Q, \,\, Qp^{-1}, \,\, Qp^{-1}Q \in L_1,
  $$
  where the derivative of the function $Q$ is
  understood in the sense of distributions. Due to a suitable
  regularization, the corresponding operators are correctly defined
  as quasi-differentials. Their resolvent convergence is
  investigated and all self-adjoint, maximal dissipative, and maximal accumulative
  extensions
  are described in terms of
  homogeneous boundary conditions of the canonical form.
\end{abstract}

\maketitle

\section{Introduction}

Many problems of mathematical physics lead to the study of Schr\"odinger-type operators with strongly singular (in particular distributional) potentials, see the monographs \cite{Albeverio,  AlbeKur}  and the more recent papers \cite{EGNST, EGNT, KM, M} and references therein.
It should be noted that the case of very  general
singular Sturm-Liouville operators defined in terms of appropriate
quasi-derivatives
has been considered in \cite{BE} (see also the book \cite{EM} and earlier discussions of quasi-derivatives in \cite{Shin, Zettl}).
Higher-order quasi-differential operators with matrix-valued valued singular coefficients were studied in \cite{Frent82, Frent99, MoZettl95, W}.

The paper \cite{S-Sh} started a new
approach for study of  one-dimensional Schr\"odinger  ope\-ra\-tors
with distributional
potential coefficients in connection with such areas as extension theory, resolvent convergence, spectral theory and inverse spectral theory.
The important  development was achieved in
\cite{GorMih2} (see also \cite{GMP12, Gor}), where it was considered the case of  Sturm-Liouville  operators generated by the differential expression
\begin{equation}
\label{vyraz} l(y) = -(py')'(t) + q(t)y(t), \quad t \in \mathcal{J}
\end{equation}
with singular distributional coefficients on a finite interval
$\mathcal{J} := (a, b)$. Namely it was assumed that
\begin{equation}\label{GM cond}
q = Q', \quad 1/p, \,\,  Q/p, \,\, Q^2/p \in L_1 (\mathcal{J}, \mathbb{C}),
\end{equation}
where the derivative of  $Q$ is understood in the sense
of distributions. The  more general class of second order quasi-differential operators was recently studied in \cite{M}.  In \cite{GMP12, GM11}  two-term singular  differential operators
\begin{equation}
\label{2term}
l(y) = i^my^{(m)}(t) + q(t)y(t), \quad t \in \mathcal{J},\quad m \ge 2
\end{equation}
with distributional coefficient $q$ were investigated. The case of matrix operators of the form (\ref{2term}) was considered in \cite{K05}.
Mention also  \cite{MS} where the deficiency indices  of  matrix  Sturm-Liouville operators with distributional coefficients on a half-line were studied.

The purpose of the present paper is to extend the results of \cite{GorMih2} to the
matrix Sturm-Liouville differential expressions.
In Section 2 we  give a
regularization of the formal differential expression (\ref{vyraz})
under a  matrix analogue of assumptions (\ref{GM
cond}). The question of norm resolvent convergence of such singular matrix  Sturm-Liouville operators is studied in Section 3.
In Section 4 we consider  the case of the symmetric minimal operator and
describe all its self-adjoint,  maximal dissipative, and maximal accumulative extensions. In addition, we study in details the case of  separated boundary conditions.

\section{Regularization of singular expression}
For positive integer $s$, denote by $M_s\equiv\mathbb{C}^{s\times s}$  the vector space of
$s\times s$ matrices with complex coefficients.
Let $\mathcal{J}:=(a,b)$ be a finite interval. Consider Lebesgue measurable matrix functions $p$, $Q$ on $\mathcal{J}$ into $M_s$ such
that $p$ is  invertible  almost everywhere. In what follows we shall always assume that
\begin{equation}\label{KM cond}
 p^{-1}, \, \,  p^{-1}Q , \, \,  Qp^{-1}, \, \,  Qp^{-1}Q \in L_1(\mathcal{J},M_s) .
\end{equation}
This condition should be considered as a matrix (noncommutative)
analogue of the assumption (\ref{GM cond}).
In particular (\ref{KM cond})  is valid under the (more restrictive)  condition
$$
 \int_{\mathcal{J}}\parallel p^{-1}(t)\parallel(1+ {\parallel Q(t)\parallel}^2)\,dt <\infty,
$$
which was (locally) assumed in the above-mentioned paper \cite{MS}.
Consider the block Shin--Zettl matrix
\begin{equation}\label{A0 matrix}
A:=\left ( \begin{array}{cc}
p^{-1} Q & p^{-1}\\
-Q p^{-1} Q & -Q
p^{-1}
\end{array}\right) \in L_1 (\mathcal{J}, M_{2s})
\end{equation}
and the corresponding quasi-derivatives
$$ D^{[0]} y = y, \quad
D^{[1]} y = py' - Qy, \quad
D^{[2]} y
=(D^{[1]} y)' + {Qp^{-1}}D^{[1]} y + {Q p^{-1}Q}y.
$$
For $q=Q'$  the Sturm-Liouville expression (\ref{vyraz}) is defined
by
\begin{equation}\label{vyraz1}l[y] := - D^{[2]} y. \end{equation}
The quasi-differential expression (\ref{vyraz1}) gives rise  to the \emph{maximal}
quasi-differential operator in the Hilbert space
$L_2\left(\mathcal{J}, \mathbb{C}^s\right) =: L_2$
 $$ L_{\operatorname{max}}:y \to
l[y],\quad \operatorname{Dom}(L_{\operatorname{max}}) := \left\{y
  \in L_2 \left| \quad  y, \,  D^{[1]}y \in AC([a,b],
    \mathbb{C}^s), D^{[2]} y \in L_2\right.\right\}.
$$
The \emph{minimal} quasi-differential operator is defined as a
restriction of the operator $L_{\operatorname{max}}$ onto the set
$$ \operatorname{Dom}(L_{\operatorname{min}})  := \left\{y \in
\operatorname{Dom}(L_{\operatorname{max}}) \left| \, \,  D^{[k]}y(a) =
D^{[k]}y(b) = 0, k = 0,1\right.\right\}.$$
 Note that under the assumption
$$p^{-1}, q \in L_1\left(\mathcal{J}, M_s\right)$$
  operators $L_{\operatorname{max}}, L_{\operatorname{min}}$
  introduced above coincide with the standard  maximal and minimal
  matrix Sturm-Liouville operators.
The regularization of the formally adjoint differential expression
$$l^+y:=-(p^*y')'(t) + q^*(t)y(t) $$
can be defined in an analogous way (here $A^*=\overline{A^T}$ is the conjugate transposed matrix to $A$).
Let $D^{\{k\}}$ ($k=0,1,2$) be the Shin--Zettl quasi-derivatives associated with $l^+$.
Denote by $L^+_{\operatorname{max}}$ and $L^+_{\operatorname{min}}$
the maximal and the minimal operators generated by this expression
on the space $L_2$. The following  results
 are proved in   \cite{Frent82} (see also \cite{MoZettl95}) in the case of general
quasi-differential matrix operators.
\begin{lemma}\label{Lagrange} {\rm(Green's formula)}. For any $y \in
\operatorname{Dom}(L_{\operatorname{max}})$,  $ z \in
\operatorname{Dom}(L^+_{\operatorname{max}})$ there holds \[ \int\limits_a^b
\left(D^{[2]}y\cdot\overline z  - y\cdot\overline{D^{\{2\}}z} \right)dt =
({D^{[1]}y\cdot\overline {z}}-{y\cdot\overline {D^{\{1\}}z}})\left|_{t = a}^{t = b}\,. \right. \]
\end{lemma}

\begin{lemma}\label{surject} For any $(\alpha _0 ,\alpha _1),(\beta _0,\beta
_1)\in\mathbb{C}^{2s}$ there exists a function ${y \in
\operatorname{Dom}(L_{\operatorname{max}})}$ such that \[ D^{[k]}y(a) = \alpha _k , \quad
D^{[k]}y(b) = \beta _k, \quad k = 0,1. \] \end{lemma}

\begin{theorem}\label{L adjoint} The operators
$L_{\operatorname{min}}$, $L^+_{\operatorname{min}}$, $L_{\operatorname{max}}$,
$L^+_{\operatorname{max}}$ are closed and densely defined in $L_2\left([a,b], \mathbb{C}^s\right)$,
and satisfy \[ L_{\operatorname{min}}^* = L^+_{\operatorname{max}},\quad L_{\operatorname{max}}^* =
L^+_{\operatorname{min}}. \] In the case of Hermitian matrices $p$ and $Q$ the operator $L_{\operatorname{min}} =
L^+_{\operatorname{min}}$ is symmetric with the deficiency indices $\left({2s,2s} \right)$, and \[
L_{\operatorname{min}}^* = L_{\operatorname{max}},\quad L_{\operatorname{max}}^* =
L_{\operatorname{min}}. \] \end{theorem}

\section{Convergence of resolvents}
Let $l_\varepsilon
[y] = -D_\varepsilon^{[2]}y$, $\varepsilon \in [0, \varepsilon_0]$,  be the quasi-differential expressions with
the coefficients
$p_\varepsilon, Q_\varepsilon$
satisfying (\ref{KM cond}).
These expressions generate the minimal operators
$L^\varepsilon_{\operatorname{min}} $,
$L^\varepsilon_{\operatorname{max}}$ in $L_2$.
Consider the quasi-differential operators
$$L_\varepsilon y =
l_\varepsilon[y],\quad \operatorname{Dom}(L_\varepsilon) =
\left\{\left.y \in
\operatorname{Dom}\left(L^\varepsilon_{\operatorname{max}}\right)\right|
\alpha(\varepsilon)\mathcal{Y}_\varepsilon(a)+\beta(\varepsilon)\mathcal{Y}_\varepsilon(b)=0\right\}.$$
Here
$\alpha(\varepsilon),\beta(\varepsilon)\in \mathbb{C}^{2s\times 2s}$
be complex matrices and
$$\mathcal{Y}_\varepsilon(a):=\left\{y(a),D^{[1]}_\varepsilon
y(a)\right\},\quad
\mathcal{Y}_\varepsilon(b):=\left\{y(b),D^{[1]}_\varepsilon
y(b)\right\}.$$
Clearly $L^\varepsilon_{\operatorname{min}} \subset L_\varepsilon
\subset L^\varepsilon_{\operatorname{max}}, \, \varepsilon \in
[0, \varepsilon_0].$
Denote by $\rho(L)$ the resolvent set of the operator $L$. Recall
that  $L_\varepsilon$ is said to converge to  $L_0$ in the norm resolvent sense, $L_\varepsilon
\stackrel{R}{\Rightarrow} L_0$, if there is a number $\mu \in \rho(L_0)$, such that  $\mu \in
\rho(L_\varepsilon)$ for all sufficiently small $\varepsilon$, and
\begin{equation}\label{R_conv}
\|(L_\varepsilon - \mu)^{-1} - (L_0 - \mu)^{-1}\| \rightarrow 0, \quad \varepsilon \rightarrow 0+.\end{equation}
It should be noted that if $L_\varepsilon
\stackrel{R}{\Rightarrow} L_0$, then the condition (\ref{R_conv}) is fulfilled for all
$\mu \in \rho(L_0)$ (see \cite{K}).

\begin{theorem}\label{res conv part}
  Suppose $\rho(L_0)$ is not empty and, for $\varepsilon \rightarrow
  0+$, the following conditions hold:
\begin{align*}
(1)&\,\, \|p^{-1}_\varepsilon - p^{-1}_0 \|_1\rightarrow 0,\\
(2)&\,\,
\|p^{-1}_\varepsilon Q_\varepsilon - p_0^{-1} Q_0\|_1 \rightarrow 0,\\
(3)&\,\, \|Q_\varepsilon p^{-1}_\varepsilon  - Q_0p_0^{-1}\|_1 \rightarrow 0,\\
(4)&\,\, \|Q_\varepsilon p^{-1}_\varepsilon Q_\varepsilon  - Q_0p_0^{-1}Q_0\|_1 \rightarrow 0,\\
(5)&\,\,\alpha(\varepsilon)\rightarrow \alpha(0),\quad
\beta(\varepsilon)\rightarrow\beta(0),
\end{align*}
where $\|\cdot\|_1$ is the norm in the space $L_1(\mathcal{J},
M_s)$. Then $L_\varepsilon \stackrel{R}{\Rightarrow} L_0$.
\end{theorem}

In essential the proof of Theorem \ref{res conv part} repeats the
arguments of \cite{GorMih2} where the scalar case $s=1$ was considered. Nevertheless the  result seems to be new even in the case of  one-dimensional Schr\"odinger operators with distributional matrix-valued potentials
($p_\varepsilon$ is the identity matrix in $\mathbb{C}^s$).
Recall the
following definition \cite{MR2}.

\begin{definition}
  Denote by $\mathcal{M}^{m}(\mathcal{J}) =:\mathcal{M}^{m},$ $ m\in
  \mathbb{N}$, the class of matrix-valued functions
$$R(\cdot;\varepsilon):[0,\varepsilon_0]\rightarrow L_1 (\mathcal{J}, {\mathbb{C}}^{m\times m})$$
parametrized by $ \varepsilon $
such that the solution of the Cauchy problem
$$
Z'(t;\varepsilon)= R( t;\varepsilon)Z(t;\varepsilon), \quad
Z(a;\varepsilon) = I,
$$
satisfies the limit condition
$$
\lim\limits_{\varepsilon \rightarrow 0+} \|Z(\cdot;\varepsilon) -
I\|_{\infty} =0,
$$
where $\|\cdot\|_{\infty}$ is the sup-norm.
\end{definition}
We need  the following  result \cite{MR2}.
\begin{theorem}\label{1 limit G}
  Suppose that the vector boundary-value problem
\begin{equation}\label{bound probl 1}
    y'(t;\varepsilon)=A(t;\varepsilon)y(t;\varepsilon)+f(t;\varepsilon),\quad
t \in \mathcal{J}, \quad \varepsilon \in [0, \varepsilon_0],
\end{equation}
\begin{equation}\label{bound probl 2}
    U_{\varepsilon}y(\cdot;\varepsilon)= 0,
\end{equation}
where the matrix-valued functions $A(\cdot,\varepsilon) \in
L_{1}(\mathcal{J}, {\mathbb{C}}^{m\times m})$, the vector-valued functions
$f(\cdot,\varepsilon) \in L_{1}(\mathcal{J}, {\mathbb{C}}^{m})$, and the linear continuous
operators
$$U_{\varepsilon}:C(\overline{\mathcal{J}};\mathbb{C}^{m}) \rightarrow\mathbb{C}^{m}, \quad m \in \mathbb{N},$$
satisfy the following conditions.
\begin{align*}
  1) \quad &\text{The homogeneous limit boundary-value
    problem~}(\ref{bound probl 1}), (\ref{bound probl 2})
  \text{~with~} \varepsilon = 0  \text{~and~} f(\cdot;0) \equiv 0 \\
  &\text{~ has only a trivial solution;}\\
  2)\quad &A(\cdot;\varepsilon)-A(\cdot;0)\in \mathcal{M}^m;\\
  3)\quad &\|U_{\varepsilon} - U_{0}\|\rightarrow 0,\quad
  \varepsilon\rightarrow 0+.
\end{align*}
Then, for a small enough $\varepsilon$, there exist Green matrices
$G(t, s; \varepsilon)$ for problems (\ref{bound probl 1}),
(\ref{bound probl 2}) and
\begin{equation}\label{G}
 \|G(\cdot,\cdot;\varepsilon)-G(\cdot,\cdot;0)\|_\infty \rightarrow 0,\quad \varepsilon\rightarrow 0+,
\end{equation}
where $\|\cdot\|_\infty$ is the norm in the space $L_\infty(\mathcal{J}\times \mathcal{J}, \, {\mathbb{C}}^{m\times m})$.
\end{theorem}

It follows from  \cite{Tamar} that conditions (1)--(4) of
Theorem~\ref{res conv part} imply
$$
A(\cdot;\varepsilon)-A(\cdot;0)\in \mathcal{M}^{2s},
$$
where the block Shin--Zettl matrix $A(\cdot;\varepsilon)$ is given by
the formula
\begin{equation}\label{A matrix}
A(\cdot;\varepsilon):=\left ( \begin{array}{cc}
p^{-1}_\varepsilon Q_\varepsilon& p^{-1}_\varepsilon\\
-Q_\varepsilon p^{-1}_\varepsilon Q_\varepsilon& -Q_\varepsilon
p^{-1}_\varepsilon
\end{array}\right).
\end{equation}
In particular $A(\cdot;0)=A$ (see (\ref{A0 matrix}).
The following two
lemmas  reduce  Theorem~\ref{res conv part} to Theorem \ref{1 limit
G}.

\begin{lemma}\label{lemm1}
  The function $y(t)$ is a solution of the boundary-value problem
\begin{equation}\label{D^2}
 l_\varepsilon[y](t)= f(t;\varepsilon) \in L_2 ,\quad\varepsilon\in [0,\varepsilon_0],
\end{equation}
\begin{equation}\label{alpha+beta}
  \alpha(\varepsilon)\mathcal{Y}_\varepsilon(a)+
  \beta(\varepsilon)\mathcal{Y}_\varepsilon(b)=0,
\end{equation}
if and only if the vector-valued function $w(t) = (y(t),
D^{[1]}_\varepsilon y(t))$ is a solution of the boundary-value
problem
\begin{equation}\label{diff eq}
w'(t)=A(t;\varepsilon)w(t) + \varphi(t;\varepsilon),
\end{equation}
\begin{equation}\label{diff alpha+beta}
\alpha(\varepsilon)w(a)+
  \beta(\varepsilon)w(b)=0,
  \end{equation}
  where the matrix-valued function $A(\cdot;\varepsilon)$ is given
  by (\ref{A matrix}) and $\varphi(\cdot;\varepsilon) := (0,
  -f(\cdot;\varepsilon))$.
\end{lemma}

\begin{proof}
Consider the system of equations
$$\left\{
\begin{aligned}
  ( D^{[0]}_\varepsilon y(t))' & = p^{-1}_\varepsilon(t) Q_\varepsilon(t)D^{[0]}_\varepsilon y(t) +
    p^{-1}_\varepsilon(t)D^{[1]}_\varepsilon y(t),\\
   ( D^{[1]}_\varepsilon y(t))' & =
   - {Q_\varepsilon(t)}p^{-1}_\varepsilon(t) Q_\varepsilon(t)D^{[0]}_\varepsilon y(t) -
    {Q_\varepsilon(t)}{p^{-1}_\varepsilon(t)}D^{[1]}_\varepsilon y(t) - f(t; \varepsilon). \\
\end{aligned}
\right. $$
Let $y(\cdot)$ be a solution of (\ref{D^2}), then the definition of
a quasi-derivative implies that $y(\cdot)$ is a solution of this
system.  On the other hand, denoting $w(t) = (D^{[0]}_\varepsilon
y(t), D^{[1]}_\varepsilon y(t))$ and $\varphi(t;\varepsilon) = (0,
-f(t;\varepsilon))$, we rewrite this system in the form of equation
(\ref{diff eq}).
Taking into account that $\mathcal{Y}_\varepsilon(a) = w(a)$,
$\mathcal{Y}_\varepsilon(b) = w(b)$, one can see that the boundary
conditions (\ref{alpha+beta}) are equivalent to the boundary
conditions (\ref{diff alpha+beta}).
\end{proof}

\begin{lemma}\label{Gamma exist}
Let a Green
matrix
 $$
 G(t,s,\varepsilon)=(g_{ij}(t,s,\varepsilon))_{i,j=1}^2\in L_\infty (\mathcal{J}\times \mathcal{J}, \, {\mathbb{C}}^{ 2s\times 2s})
 $$
 exist for the problem (\ref{diff eq}), (\ref{diff alpha+beta}) for
 small enough $\varepsilon$. Then there exists a Green function
 $\Gamma(t,s;\varepsilon)$ for the semi-homogeneous boundary-value
 problem (\ref{D^2}), (\ref{alpha+beta}) and
$$ \Gamma(t,s;\varepsilon) = -g_{12}(t,s;\varepsilon)
\quad\mbox{a.e.}$$
\end{lemma}

\begin{proof}
  According to the definition of a Green matrix, a unique solution
  of the problem (\ref{diff eq}), (\ref{diff alpha+beta}) can be written
  in the form
$$w_\varepsilon(t)=\int\limits_a ^b
G(t,s;\varepsilon)\varphi(s;\varepsilon) ds, \quad t\in \mathcal{J}.$$

Due to Lemma \ref{lemm1}, the latter equality can be rewritten in the
form
$$\left\{
\begin{aligned}
    D^{[0]}_\varepsilon y_\varepsilon(t) & = \int\limits_a^b g_{12}(t,s;\varepsilon)(-f(s;\varepsilon))\,ds, \\
    D^{[1]}_\varepsilon y_\varepsilon(t) & = \int\limits_a^b g_{22}(t,s;\varepsilon)(-f(s;\varepsilon))\,ds, \\
\end{aligned}
\right.$$ where $y_\varepsilon(\cdot)$ is a unique solution of
(\ref{D^2}), (\ref{alpha+beta}).  This implies the statement of
Lemma \ref{Gamma exist}.
\end{proof}

\begin{proof}[Proof of Theorem \ref{res conv part}]
Consider  matrices $$Q_{\varepsilon(t),\mu}=Q_\varepsilon(t)+\mu tI,
\,\, p_{\varepsilon(t),\mu}=p_\varepsilon(t)$$  corresponding to the
operators $L_\varepsilon + \mu I$. Clearly assumption (\ref{KM
cond}) and conditions (1)--(4) of Theorem \ref{res conv part}  do
not depend on $\mu$ and we can assume without loss of generality
that $0 \in \rho(L_0)$. It follows that the homogeneous boundary-value problem
  $$l_0[y](t)=0, \quad \alpha(0)\mathcal{Y}_0(a)+
  \beta(0)\mathcal{Y}_0(b) = 0$$ has only a trivial
  solution. Due to Lemma \ref{lemm1} the homogeneous boundary-value problem
$$w'(t)=A(t;0)w(t), \quad \alpha(0)w(a) + \beta(0)w(b)=0$$
also has only a trivial solution.
By conditions (1)--(4) of Theorem \ref{res conv part} we have
 that $A(\cdot;\varepsilon)-A(\cdot;0)\in \mathcal{M}^{2s}$,
where $A(\cdot;\varepsilon)$ is given by formula (\ref{A matrix}).
Thus statement of Theorem \ref{res conv part} implies that the
problem (\ref{diff eq}), (\ref{diff alpha+beta}) satisfies
conditions of Theorem \ref{1 limit G}. It follows that Green
matrices $G(t,s;\varepsilon)$ of the problems (\ref{diff eq}),
(\ref{diff alpha+beta}) exist.  Taking into account Lemma \ref{Gamma exist}  and  (\ref{G}) we have that
\begin{align*}
\|L_\varepsilon^{-1}- L_0^{-1}\| & \leq \|L_\varepsilon^{-1}- L_0^{-1}\|_{HS}=\|\Gamma(\cdot,\cdot;\varepsilon) -
\Gamma(\cdot,\cdot;0)\|_2 \\
& \leq (b -
a)\|\Gamma(\cdot,\cdot;\varepsilon) -
\Gamma(\cdot,\cdot;0)\|_\infty \rightarrow 0, \quad \varepsilon
\rightarrow 0+.
\end{align*}
Here $\|\cdot\|_{HS}$ is the Hilbert-Schmidt norm.
\end{proof}
\begin{remark}\label{HS}
It follows from the proof that   $(L_{\varepsilon}-\mu)^{-1} \rightarrow (L_{0}-\mu)^{-1}$ in a
Hilbert-Schmidt norm for all $\mu \in \rho (L_0)$.
\end{remark}

\section{Extensions of symmetric minimal operator}\label{section_ext}
In what follows we additionally suppose that the matrix functions $p$, $Q$ and,
consequently, the distribution $q = Q'$ to be Hermitian.
By Theorem \ref{L adjoint} the minimal ope\-rator
$L_{\operatorname{min}}$ is symmetric and one may consider  a problem of
describing (in terms of homogeneous boundary conditions) all
self-adjoint, maximal dissipative, and maximal accumulative extensions of the operator $L_{\operatorname{min}}$.
Let us recall following definition.
\begin{definition}
  Let $L$ be a closed densely defined symmetric operator in a Hilbert space
  $\mathcal{H}$ with equal (finite or infinite) deficient indices.
  The triplet $\left( {H,\Gamma _1 ,\Gamma _2 }\right)$, where $H$
  is an auxiliary Hilbert space and $\Gamma_1$, $\Gamma_2$ are the
  linear mappings of $\operatorname{Dom}(L^*)$ onto $H,$ is called
  a \emph{boundary triplet} of the symmetric operator $L$, if
\begin{enumerate}
  \item  for any $ f,g \in \operatorname{Dom} \left( {L^*} \right)$,
  $$
\left( {L^ *  f,g} \right)_\mathcal{H} - \left( {f,L^ *  g}
\right)_\mathcal{H} = \left( {\Gamma _1 f,\Gamma _2 g} \right)_H  -
\left( {\Gamma _2 f,\Gamma _1 g} \right)_H,
$$
\item for any $ f_1, f_2 \in H$ there is a vector $ f\in
  \operatorname{Dom} \left( {L^*} \right)$ such that $\Gamma _1 f =
  f_1$, $ \Gamma _2 f = f_2$.
\end{enumerate}
\end{definition}

The definition of a boundary triplet implies that $ f \in
\operatorname{Dom} \left( {L} \right)$ if and only if $\Gamma_1f =
\Gamma_2f = 0$. A boundary triplet exists for any symmetric operator
with equal non-zero deficient indices (see \cite{Gorbachuk} and
references therein).
The following
result is crucial for the rest of the paper.

\begin{lemma}\label{baslemm}
  \label{PGZth} Triplet $(\mathbb{C}^{2s}, \Gamma_1, \Gamma_2)$,
  where $\Gamma_1, \Gamma_2$ are the linear mappings
$$
\Gamma_1y := \left( D^{[1]}y(a), -D^{[1]}y(b)\right), \quad \Gamma_2y := \left( y(a), y(b)\right),
$$
from $\operatorname{Dom}(L_{\operatorname{max}})$ onto
$\mathbb{C}^{2s}$ is a boundary triplet for the operator
$L_{\operatorname{min}}$.
\end{lemma}

\begin{proof}
  According to
  Theorem \ref{L adjoint}, $L^*_{\operatorname{min}} = L_{\operatorname{max}}$.
  Due to Lemma \ref{Lagrange},
 $$\left( {L_{\operatorname{max}}y,z}
  \right) - \left( y,L_{\operatorname{max}}z\right) =
  \left.\left(y\cdot\overline{D^{[1]}z} -
      D^{[1]}y\cdot\overline{z}\right)\right|^b_a.$$
  But
\begin{align*}
\left( {\Gamma _1 y,\Gamma _2 z} \right) & =
D^{[1]}y(a)\cdot\overline{z(a)} -
D^{[1]}y(b)\cdot\overline{z(b)},\\ \left({\Gamma _2
y,\Gamma _1 z} \right) & = y(a)\cdot\overline{D^{[1]}z(a)}
- y(b)\cdot\overline{D^{[1]}z(b)}.
\end{align*}
This means that condition $1)$ is fulfilled. Condition $2)$ is true due to Lemma \ref{surject}.
\end{proof}
Let $K$ be a linear operator on  $\mathbb{C}^{2s}.$
Denote by $L_K$ the restriction of
  $L_{\operatorname{max}}$ onto the set of functions ${y \in
    \operatorname{Dom}(L_{\operatorname{max}})}$ satisfying the
  homogeneous boundary condition in the canonical form
\begin{equation} \label{rozsh}
 \left( {K - I} \right)\Gamma _1 y +
i\left( {K + I} \right)\Gamma _2 y = 0.
\end{equation}
Similarly,   $L^K$ denotes the restriction of
$L_{\operatorname{max}}$ onto the set of the functions $y \in
\operatorname{Dom}(L_{\operatorname{max}})$ satisfying the boundary condition \begin{equation}
\label{akk_ext}
 \left( K - I \right)\Gamma_{1} y - i\left( K + I \right)\Gamma_{2} y = 0.
\end{equation}
Clearly, $L_K$ and $L^K$ are the extensions of $L$ for any $K$.
Recall that a densely
defined linear operator $T$ on a complex Hilbert space $\mathcal{H}$ is called \emph{dissipative}
(resp. \emph{accumulative}) if \[ \Im \left( Tx, x \right)_\mathcal{H} \geq 0 \quad (\text{resp.}
\leq 0), \quad \text{for all}\quad  x \in \text{Dom} (T) \] and it is called \emph{maximal dissipative}
(resp. \emph{maximal accumulative}) if, in addition, $T$ has no non-trivial
dissipative (resp. accumulative) extensions in $\mathcal{H}.$ Every symmetric ope\-ra\-tor is both dissipative
and accumulative, and every self-adjoint operator is a maximal dissipative and maximal accumulative one.
Lemma \ref{baslemm} together with results of \cite[Ch.~3]{Gorbachuk} leads to
the following description of  dissipative, accumulative and self-adjoint extensions of
$L_{\operatorname{min}}$.

\begin{theorem}\label{adj ext} Every $L_K$ with $K$ being a contracting operator in
$\mathbb{C}^{2s}$,
 is a maximal dissipative extension of   $L_{\operatorname{min}}$.
Similarly every $L^K$ with $K$ being a contracting operator in $\mathbb{C}^{2s}$, is a maximal
accumulative extension of the operator $L_{\operatorname{min}}$. Conversely, for any maximal
dissipative (respectively, maximal accumulative) extension $\widetilde{L}$ of the operator
$L_{\operatorname{min}}$ there exists a contracting operator $K$ such that $\widetilde{L} =
L_K$\,\, (respectively, $\widetilde{L} = L^K$). The extensions $L_K$ and $L^K$ are self-adjoint if
and only if $K$ is a unitary operator on $\mathbb{C}^{2s}$. These correspondences between operators
$\{K\}$ and the extensions $\{\widetilde{L}\}$ are all bijective.
\end{theorem}

\begin{remark}\label{homeo} It follows from
  Theorem \ref{res conv part} and Theorem \ref{adj ext}
   that the mapping $K \to L_K$ is not only bijective but
  also continuous. More accurately, if  contracting operators $K_n$
  converge to an operator $K$, then  $L_{K_n}
\stackrel{R}{\Rightarrow} L_K.$ The converse is
  also true, because the set of contracting operators in the space
  $\mathbb{C}^{2s}$ is a compact set. This means that the mapping
  $$K
  \to \left(L_K - \lambda\right)^{-1}, \quad \operatorname{Im}
  \lambda < 0,$$ is a homeomorphism for any fixed $\lambda.$ Analogous result  is true for $L^K.$
\end{remark}

Now we pass to the description of separated boundary
conditions. Denote by ${f_a}$ the germ of a continuous function $f$
at the point $a$.

\begin{definition}
  The boundary conditions  that define the operator $L \subset
  L_{\operatorname{max}}$ are called \emph{separated} if for
  arbitrary functions $y \in \operatorname{Dom}(L)$ and any $g, h \in
  \operatorname{Dom}(L_{\operatorname{max}})$, such that
$${g_a} = {y_a},\quad  {g_b} = 0,\quad
{h_a} = 0,\quad  {h_b} = {y_b}$$ we have $g, h \in
\operatorname{Dom}(L).$
\end{definition}

\begin{theorem}\label{divided adj}
Let $K$ be a linear operator on  $\mathbb{C}^{2s}.$
Boundary conditions (\ref{rozsh}), (\ref{akk_ext}) defining $L_K$ and $L^K$ respectively are separated if and
only if  $K$ is block diagonal, i.e.,
\begin{equation} \label{separable cond}
K = \left(%
\begin{array}{cc}
  K_a & 0 \\
  0 & K_b \\
\end{array}%
\right),
 \end{equation}
where $K_a, K_b$ are arbitrary $s \times s$ matrices.
\end{theorem}

\begin{proof} We consider the operators $L_K$, the case of $L^K$ can be treated in a
similar way.
The assumption ${y_c} = {g_c}$ implies that
\begin{equation}\label{separated def}
y(c) = g(c), \quad (D^{[1]}y)(c) = (D^{[1]}g)(c), \quad c \in [a, b].
\end{equation}
Let  $K$ have the form (\ref{separable cond}).  Then (\ref{rozsh}) can be
written in the form of a system,
\[ \left\lbrace
\begin{aligned}
(K_a - I)D^{[1]}y(a) + i(K_a + I)y(a)& = 0,
\\
 -(K_b - I)D^{[1]}y(b)
+ i(K_b + I)y(b) & = 0.
\end{aligned}
 \right. \]
Clearly these conditions are separated.
Conversely, suppose that boundary conditions   (\ref{rozsh}) are
separated. The matrix $K \in \mathbb{C}^{2s \times 2s}$ can be
written in the form
$$K = \left(%
\begin{array}{cc}
  K_{11} & K_{12} \\
  K_{21} & K_{22} \\
\end{array}%
\right).$$
We need to prove that $K_{12} = K_{21} = 0$.
Let us rewrite (\ref{rozsh}) in the form of
the system
\[ \left\lbrace
\begin{aligned}
( K_{11} - I)D^{[1]}y(a) - K_{12}D^{[1]}y(b) + i( K_{11} + I)y(a) +
iK_{12}y(b) & = 0,\\ K_{21}D^{[1]}y(a) - (K_{22} - I)D^{[1]}y(b) +
iK_{21}y(a) + i(K_{22} + I)y(b)&  = 0.
\end{aligned}
 \right. \]
The fact that the boundary conditions are separated implies that a
 function $g$ such that ${g_a} = {y_a}, {g_b} =
 0$ also satisfies this system.
It follows from (\ref{separated def})  that for any $y \in \operatorname{Dom}(L_K)$
\[ \left\lbrace
\begin{aligned}
&K_{11} \left[D^{[1]}y(a) + iy(a)\right]= D^{[1]}y(a) - iy(a),\\
&K_{21}\left[D^{[1]}y(a) + iy(a)\right] = 0 .
\end{aligned}
 \right. \]
This means that for any $y \in \operatorname{Dom}(L_K)$
\begin{equation} \label{KerK21}
D^{[1]}y(a)
+ iy(a) \in \Ker (K_{21}).
\end{equation}
 For any $z =
(z_1, z_2) \in \mathbb{C}^{2s}$, consider the vectors $-i \left(K +
  I\right)z$ and $\left(K - I\right)z$. Due to  Lemma \ref{baslemm} and
the definition of the boundary triplet, there exists a function $y_z
\in \operatorname{Dom}(L_{\operatorname{max}})$ such that
\begin{equation}\label{rozsh parametric}
\left\lbrace
\begin{aligned}
-i \left(K + I\right)z &  = \Gamma _1 y_z,\\ \left(K - I\right)z & =
\Gamma _2 y_z.
\end{aligned}\right.
\end{equation}
Clearly $y_z$ satisfies  (\ref{rozsh}) and ${y_z \in}$
${\operatorname{Dom}(L_K)}$.  Rewrite (\ref{rozsh
  parametric}) in the form of the system
\[
\left\lbrace
\begin{aligned}
- i(K_{11} + I)z_1 - iK_{12}z_2 & = D^{[1]}y_z(a),
\\ -iK_{21}z_1 -
i(K_{22} + I)z_2 & = -D^{[1]}y_z(b),
\\ (K_{11} - I)z_1 + K_{12}z_2 & =
y_z(a),
\\ K_{21}z_1 + (K_{22} - I)z_2 & = y_z(b).
\end{aligned}
\right.
\]
The first and the third equations of the system above imply that for
any $z_1 \in \mathbb{C}^s$ $$ D^{[1]}y_z(a) + iy_z(a) = -2iz_1.$$
Due to (\ref{KerK21}) we have that $\Ker(K_{21}) = \mathbb{C}^s$ and
therefore $K_{21}=0$. Similarly one can prove that $K_{12} = 0$.
\end{proof}

\begin{remark}
It follows from Lemma \ref{baslemm} and Theorem 1 of \cite{Brook} that
there is a one-to-one correspondence between the generalized
  resolvents $R_\lambda$ of  $L_{\operatorname{min}}$ and the
  boundary-value problems
  $$ l[y] = \lambda y + h, \,\,  \left(
    {K(\lambda) - I} \right)\Gamma _1 y + i\left( {K(\lambda) + I}
  \right)\Gamma _2 y = 0. $$ Here
  $\operatorname{Im}\lambda < 0$, $h \in L_2$, and $K(\lambda)$
  is an operator-valued function on the space $\mathbb{C}^{2s}$,
  regular in the lower half-plane, such that $||K(\lambda)|| \leq
  1$.  This correspondence is given by the identity $$R_\lambda h =
  y,\quad \operatorname{Im}\lambda < 0.$$
\end{remark}




\begin{thebibliography}{10}

\bibitem{Albeverio}
S.~Albeverio, F.~Gesztesy, R.~H{\o}egh-Krohn, and H.~Holden, \emph{Solvable
  models in quantum mechanics}, Texts and Monographs in Physics,
  Springer-Verlag, New York, 1988.

\bibitem{AlbeKur}
S.~Albeverio and P.~Kurasov, \emph{Singular perturbations of differential
  operators}, London Mathematical Society Lecture Note Series, vol. 271,
  Cambridge University Press, Cambridge, 2000.

\bibitem{BE}
C.~Bennewitz and W.N. Everitt, \emph{On second-order left-definite boundary
  value problems}, Ordinary differential equations and operators ({D}undee,
  1982), Lecture Notes in Math., vol. 1032, Springer, Berlin, 1983, pp.~31--67.

\bibitem{Brook}
V.M. Bruk, \emph{A certain class of boundary value problems with a spectral
  parameter in the boundary condition}, Mat. Sb. (N.S.) \textbf{100 (142)}
  (1976), no.~2, 210--216.

\bibitem{EGNST}
J.~Eckhardt, F.~Gesztesy, R.~Nichols, A.~Sakhnovich, and G.~Teschl,
  \emph{Inverse spectral problems for {S}chr\"odinger-type operators with
  distributional matrix-valued potentials}, Differential Integral Equations
  \textbf{28} (2015), no.~5-6, 505--522.

\bibitem{EGNT}
J.~Eckhardt, F.~Gesztesy, R.~Nichols, and G.~Teschl, \emph{Supersymmetry and
  {S}chr\"odinger-type operators with distributional matrix-valued potentials},
  J. Spectr. Theory \textbf{4} (2014), no.~4, 715--768.

\bibitem{EM}
W.N. Everitt and L.~Markus, \emph{Boundary value problems and symplectic
  algebra for ordinary differential and quasi-differential operators},
  Mathematical Surveys and Monographs, vol.~61, American Mathematical Society,
  Providence, RI, 1999.

\bibitem{Frent82}
H.~Frentzen, \emph{Equivalence, adjoints and symmetry of quasidifferential
  expressions with matrix-valued coefficients and polynomials in them}, Proc.
  Roy. Soc. Edinburgh Sect. A \textbf{92} (1982), no.~1-2, 123--146.

\bibitem{Frent99}
H.~Frentzen, \emph{Quasi-differential operators in {$L\sp p$} spaces}, Bull.
  London Math. Soc. \textbf{31} (1999), no.~3, 279--290.

\bibitem{Gorbachuk}
V.I. Gorbachuk and M.L. Gorbachuk, \emph{Boundary value problems for operator
  differential equations}, Mathematics and its Applications, vol.~48, Springer
  Netherlands, 1991.

\bibitem{GorMih2}
A.~Goriunov and V.~Mikhailets, \emph{Regularization of singular
  {S}turm-{L}iouville equations}, Methods Funct. Anal. Topology \textbf{16}
  (2010), no.~2, 120--130.

\bibitem{GMP12}
A.~Goriunov, V.~Mikhailets, and K.~Pankrashkin, \emph{Formally self-adjoint
  quasi-differential operators and boundary-value problems}, Electron. J.
  Differential Equations (2013), no.~101, 1--16.

\bibitem{GM11}
A.S. Goryunov and V.A. Mikhailets, \emph{Regularization of two-term
  differential equations with singular coefficients by quasiderivatives},
  Ukrainian Math. J. \textbf{63} (2012), no.~9, 1361--1378.

\bibitem{Gor}
A.S. Horyunov, \emph{Convergence and approximation of the {S}turm-{L}iouville
  operators with potentials-distributions}, Ukrainian Math. J. \textbf{67}
  (2015), no.~5, 680--689.

\bibitem{K}
T.~Kato, \emph{Perturbation theory for linear operators}, Classics in
  Mathematics, Springer-Verlag, Berlin, 1995.

\bibitem{MR2}
T.I. Kodlyuk, V.A. Mikhailets, and N.V. Reva, \emph{Limit theorems for
  one-dimensional boundary-value problems}, Ukrainian Math. J. \textbf{65}
  (2013), no.~1, 77--90.

\bibitem{K05}
O.O. Konstantinov, \emph{Two-term differential equations with matrix
  distributional coefficients}, Ukrainian Math. J. \textbf{67} (2015), no.~5,
  711--722.

\bibitem{KM}
A.S. Kostenko and M.M. Malamud, \emph{1-{D} {S}chr\"odinger operators with
  local point interactions on a discrete set}, J. Differential Equations
  \textbf{249} (2010), no.~2, 253--304.

\bibitem{M}
K.A. Mirzoev, \emph{Sturm-{L}iouville operators}, Trans. Moscow Math. Soc.
  (2014), 281--299.

\bibitem{MS}
K.A. Mirzoev and T.A. Safonova, \emph{On the deficiency index of the
  vector-valued {S}turm-{L}iouville operator}, Math. Notes \textbf{99} (2016),
  no.~2, 290--303.

\bibitem{MoZettl95}
M.~M{\"o}ller and A.~Zettl, \emph{Semi-boundedness of ordinary differential
  operators}, J. Differential Equations \textbf{115} (1995), no.~1, 24--49.

\bibitem{S-Sh}
A.M. Savchuk and A.A. Shkalikov, \emph{Sturm-{L}iouville operators with
  singular potentials}, Math. Notes \textbf{66} (1999), no.~6, 741--753.

\bibitem{Shin}
D.~Shin, \emph{Quasi-differential operators in {H}ilbert space}, Mat. Sb.
  \textbf{13 (55)} (1943), 39--70 (in Russian).

\bibitem{Tamar}
J.D. Tamarkin, \emph{A lemma of the theory of linear differential systems},
  Bull. Amer. Math. Soc. \textbf{36} (1930), no.~2, 99--102.

\bibitem{W}
J.~Weidmann, \emph{Spectral theory of ordinary differential operators}, Lecture
  Notes in Mathematics, vol. 1258, Springer-Verlag, Berlin, 1987.

\bibitem{Zettl}
A.~Zettl, \emph{Formally self-adjoint quasi-differential operators}, Rocky
  Mountain J. Math. \textbf{5} (1975), 453--474.

\end{thebibliography}
\end{document}